\documentclass{amsart} 
 
\usepackage{amsmath, amsthm, amssymb}
\usepackage{mathrsfs, bm}
\usepackage{hyperref}
\usepackage[pdftex]{graphicx}
\usepackage{graphics}

\DeclareSymbolFont{largesymbolsTXA}{U}{ntxexa}{m}{n}
\SetSymbolFont{largesymbolsTXA}{bold}{U}{ntxexa}{b}{n}
\DeclareFontSubstitution{U}{ntxexa}{m}{n}
\DeclareMathSymbol{\fintop}{\mathop}{largesymbolsTXA}{62}

\newcommand{\RR}{\mathbb{R}}

\renewcommand{\SS}{\mathbb{S}}

\newcommand{\al}{\alpha}
\newcommand{\be}{\beta}

\newcommand{\ep}{\epsilon}

\newcommand{\la}{\lambda}
\newcommand{\om}{\omega}

\newcommand{\Si}{\Sigma}

\newcommand{\pa}{\partial}

\newcommand{\proofofmaintheorem}{\par{\noindent\textit{Proof of Theorem 2.3. }}}

\newtheorem{theorem}{Theorem}[section]
\newtheorem{lemma}[theorem]{Lemma}
\newtheorem{remark}[theorem]{Remark}
\newtheorem{proposition}[theorem]{Proposition}
\newtheorem{corollary}[theorem]{Corollary}
\newtheorem{conjecture}[theorem]{Conjecture}

\title{On the Existence of Stable of Unduloids of Dimension Eight}

\author{David Hartley}

\address{University of Wollongong, Northfields Avenue, Wollongong 2522}

\email{hartleyd@uow.edu.au}

\date{\today}

\begin{document}

\begin{abstract}
In this paper we study the stability of $n$-dimensional constant mean curvature unduloids embedded in slabs in $\RR^{n+1}$. We prove that among the family of half period unduloids stability is determined by whether the volume is increasing or decreasing along this family provided some conditions on the volume function are met. We then use this theorem to prove the existence of stable unduloids of dimension eight.
\end{abstract}

\maketitle

\section{Introduction}\label{SecIntro}
We consider the isoperimetric problem for hypersurfaces, $\Si$, embedded in a slab, $M=[0,d]\times\RR^n\subset\RR^{n+1}$, such that $\pa\Si\subset\pa M$. The free boundary critical points of the area function under a volume constraint are the constant mean curvature (CMC) hypersurfaces that meet $\pa M$ orthogonally (we will refer to this as the free boundary condition). It is further known that these hypersurfaces must be rotationally symmetric, \cite{Athanassenas87,Pedrosa99}, thus they are the Delaunay hypersurfaces, consisting of catenoids, spheres, cylinders, unduloids, and nodoids \cite{Delaunay41,Hsiang81}. We will reject the catenoids and nodoids as the former cannot satisfy the boundary conditions, while the later can only satisfy them after leaving the slab.

To investigate whether these hypersurfaces minimise the area (locally under a volume constraint), we introduce the concept of stability. A constant mean curvature hypersurface $\Si\subset\RR^{n+1}$ is stable if the functional
\begin{equation}\label{StabilityGen}
J_{\Si}(u)=\frac{\int \|\nabla u\|_{\Si}^2-|A|^2u^2\,d\mu_{\Si}}{\int u^2\,d\mu_{\Si}}
\end{equation}
is non-negative for all functions $u$ on $\Si$ such that $\int u\,d\mu_{\Si}=0$, where $|A|$ is the magnitude of the second fundamental form of $\Si$. When the functional is non-negative for all functions $u$, we call the hypersurface strictly stable. This functional is the second variation of area under the volume constraint and as such any solution to the isoperimetric problem is stable.

Spheres are known to be stable hypersurfaces \cite{Barbosa84} and, in fact, the only stable compact orientable immersion, while half spheres in the slab are also stable by the same reasoning \cite{Athanassenas87}. A cylinder is stable if and only if its radius is greater than or equal to $\frac{d\sqrt{n-1}}{\pi}$, this was proved in the $n=2$ case in \cite{Athanassenas87,Vogel87} and in general dimensions in \cite{Hartley13,Souam18}, with the former paper doing so in with respect to the stronger condition of dynamic stability of the volume preserving mean curvature flow. The stability of unduloids is a more complicated topic. Pedrosa and Ritor\'e, \cite{Pedrosa99}, proved that if $2\leq n\leq 7$ then all unduloids are unstable (they also proved that nodoids are unstable), and that the near spherical unduloids are unstable. However, they also showed in that paper that if $n\geq9$ there exist stable unduloids. Their proof of existence was based on a comparison of the area of a degenerate half sphere (where its apex touches a boundary of the slab) with the area of the cylinder of the same volume. If $n\geq9$ then the area of the half sphere is smaller, meaning the cylinder does not solve the isoperimetric problem at this volume. Since the half sphere is degenerate it does not either, hence an unduloid must solve the isoperimetric problem at this volume and hence be stable. More recent results \cite{Hartley16,Li18} considered the near cylindrical unduloids and showed that the near cylindrical unduloids of half period are stable if and only if $n\geq11$, with the former paper again doing so in the context of dynamic stability. Any unduloid with over a half a period is unstable, as it is no longer a graph over a boundary component of $S$.

These results have two major gaps. Do their exist stable unduloids of dimension eight? And what characteristic determines the stability of unduloids? In this paper we provide a positive answer to the first question, while giving an answer to the second subject to some conditions.

The paper is organised as follows. In Section \ref{SecFamily} we introduce a family of unduloids and with this are able to state our main theorem. In Section \ref{SecStable} we consider the functional (\ref{StabilityGen}) and recast the concept of stability in terms of an operator. In Section \ref{SecNull} we examine the null space of this operator and in Section \ref{SecVary} we consider how a zero eigenvalue will vary along the family of unduloids, this allows us to prove the main theorem. Sections \ref{SecFamilyApp}, \ref{SecStableApp}, and \ref{PsiSec} contain the details of some technical calculations.

\textbf{Acknowledgement}: The author would like to thank Professor Frank Morgan for his question regarding the isoperimetric problem, which prompted the inclusion of Remark \ref{MorganRem}.


\section{A Family of Rotationally Symmetric CMC Hypersurfaces}\label{SecFamily}
In this section we give a representation of unduloids in terms of a profile curve and use the characteristics of this family to state the main stability theorem. We start by defining a couple of intermediate functions that will appear in our family. The first is a function of the parameter that is used to set the period of the hypersurface:
\[
Q(t):=\left\{\begin{array}{cc}\frac{1-t^{n-1}}{1-t^n}, & t\in\RR_+\backslash\{1\},\\ \frac{n-1}n, & t=1,\end{array}\right.
\]
note that $Q$ is continuously differentiable and strictly decreasing for $t>0$ with
\[
Q'(t)=\left\{\begin{array}{cc}-\frac{t^{n-2}\left(t^n-nt+n-1\right)}{\left(1-t^n\right)^2}, & t\in\RR_+\backslash\{1\},\\ -\frac{n-1}{2n}, & t=1,\end{array}\right.
\]
and $Q(t^{-1})=tQ(t)$, so $tQ(t)$ is strictly increasing. The second function we define is the gradient of the profile curve:
\[
R(x;t):=\left\{\begin{array}{cc}\frac{1}{|1-t|}\sqrt{\left(\frac{(1-(1-t)x)^{n-1}}{1-Q(t)+Q(t)(1-(1-t)x)^n}\right)^2-1}, & t\in\RR_+\backslash\{1\},\\ \sqrt{(n-1)(1-x)x}, & t=1,\end{array}\right.
\]
where $0\leq x\leq1$. We again note that this function is continuous for $t>0$ and satisfies $R(0;t)=R(1;t)=0$ for all $t>0$, along with $R(1-x;t^{-1})=tR(x;t)$. From these functions we further define
\begin{equation}\label{Defzeta}
\zeta(x;t):=\left\{\begin{array}{cc}\int_x^1R(\tilde{x};t)^{-1}\,d\tilde{x}, & 0\leq x<1,\\ 0, & x=1,\end{array}\right.
\end{equation}
\begin{equation}\label{DefP}
P(t):=\zeta(0;t)=\int_0^1R(x;t)^{-1}\,dx,
\end{equation}
and let $\zeta^{-1}(y;t)$ be the inverse of $\zeta$ with respect to its first variable, that is $\zeta^{-1}(\zeta(x;t);t)=x$ for all $x$ and $t$. Also note that $P(t^{-1})=t^{-1}P(t)$. Throughout the paper we will use subscripts to denote differentiation of functions with respect to that variable.

We can now define the family of CMC hypersurfaces, we save the calculations for Section \ref{SecFamilyApp}.
\begin{proposition}\label{FamCMC}
The two parameter family of functions
\[
u(z;r,t):=\frac{Q(r)d}{P(t)Q(t)}\left(1-(1-r)\zeta^{-1}\left(\frac{Q(t)P(t)z}{Q(r)d};r\right)\right),
\]
for $r,t>0$ and $z\in\left[0,\frac{Q(r)P(r)d}{Q(t)P(t)}\right]$ define a two parameter family of rotationally symmetric CMC hypersurfaces with mean curvature
\[
\eta(t):=\frac{nQ(t)P(t)}{d},
\]
and such that
\[
u_z(0;r,t)=u_z\left(\frac{Q(r)P(r)d}{Q(t)P(t)};r,t\right)=0.
\]
\end{proposition}

We will consider the finite unduloids of length $d$ and fixed period $2d$, these are given by the profile curve
\[
v(z;t):=u(z;t,t)=\frac{d}{P(t)}\left(1-(1-t)\zeta^{-1}\left(\frac{P(t)z}{d};t\right)\right),
\]
have a mean curvature $\eta(t)$, and have an $(n+1)$-enclosed volume $V(t):=Vol(v(\cdot;t))$.

\begin{remark}
Note that $\eta(t^{-1})=\eta(t)$. It was this function that was used in \cite{Pedrosa99} to prove their instability results, although they used a different parameterisation and considered the family of with the same mean curvature (not same period). In the notation of this paper, they proved that if $\eta'(t)>0$ for some $t\in(0,1)$ then $v(\cdot;t)$ defines an unstable unduloid.
\end{remark}

The main theorem of this paper is:
\begin{theorem}\label{stability}
Let $n\geq 2$ be such that if $t\in(0,1)$ is such that $V'(t)=0$ then $V''(t)\neq0$ and $\eta'(t)<0$. Let $t_0\in(0,1]$, if $V'(t_0)<0$ the CMC unduloid defined by $v(\cdot;t_0)$ is unstable. While when $V'(t_0)>0$ the CMC unduloid defined by $v(\cdot;t_0)$ is stable.
\end{theorem}

\begin{remark}
This also covers the $t_0>1$ case by symmetry of the hypersurfaces under the transformation $t\to t^{-1}$. That is, if $t_0>1$ and $V'(t_0)<0$, then the CMC unduloid defined by $v(z;t_0)$ is stable, while if $V'(t_0)>0$ it is unstable.
\end{remark}

\begin{remark}
It is conjectured that the conditions $V''(t)\neq0$ and $\eta'(t)<0$ whenever $t\in(0,1)$ is such that $V'(t)=0$, are satisfied in all dimensions $n\geq2$, however the relevant integral bounds have not been obtained. In fact, we have the following conjecture which is stronger than the second condition.
\end{remark}

\begin{conjecture}
The function $\xi(t):=\eta(t)^{n+1}V(t)$ is strictly decreasing on $(0,1)$ for $n\geq2$.
\end{conjecture}

\begin{corollary}\label{Dim8}
There are stable CMC unduloids of dimension eight.
\end{corollary}

\begin{proof}
In Figure \ref{Fig1} the functions $\frac{V'(t)}{1-t}$ (blue), $V''(t)$ (orange), and $\xi'(t)$ (green), normalised so that the maximum of their absolute values over the domain $(0,1)$ is $1$, for $n=8$ are plotted. $V'(t)$ has been divided by $1-t$ to remove the zero at $t=1$, which is not relevant to the discussion.

\begin{figure}[h]
\centering
\includegraphics[width=\textwidth]{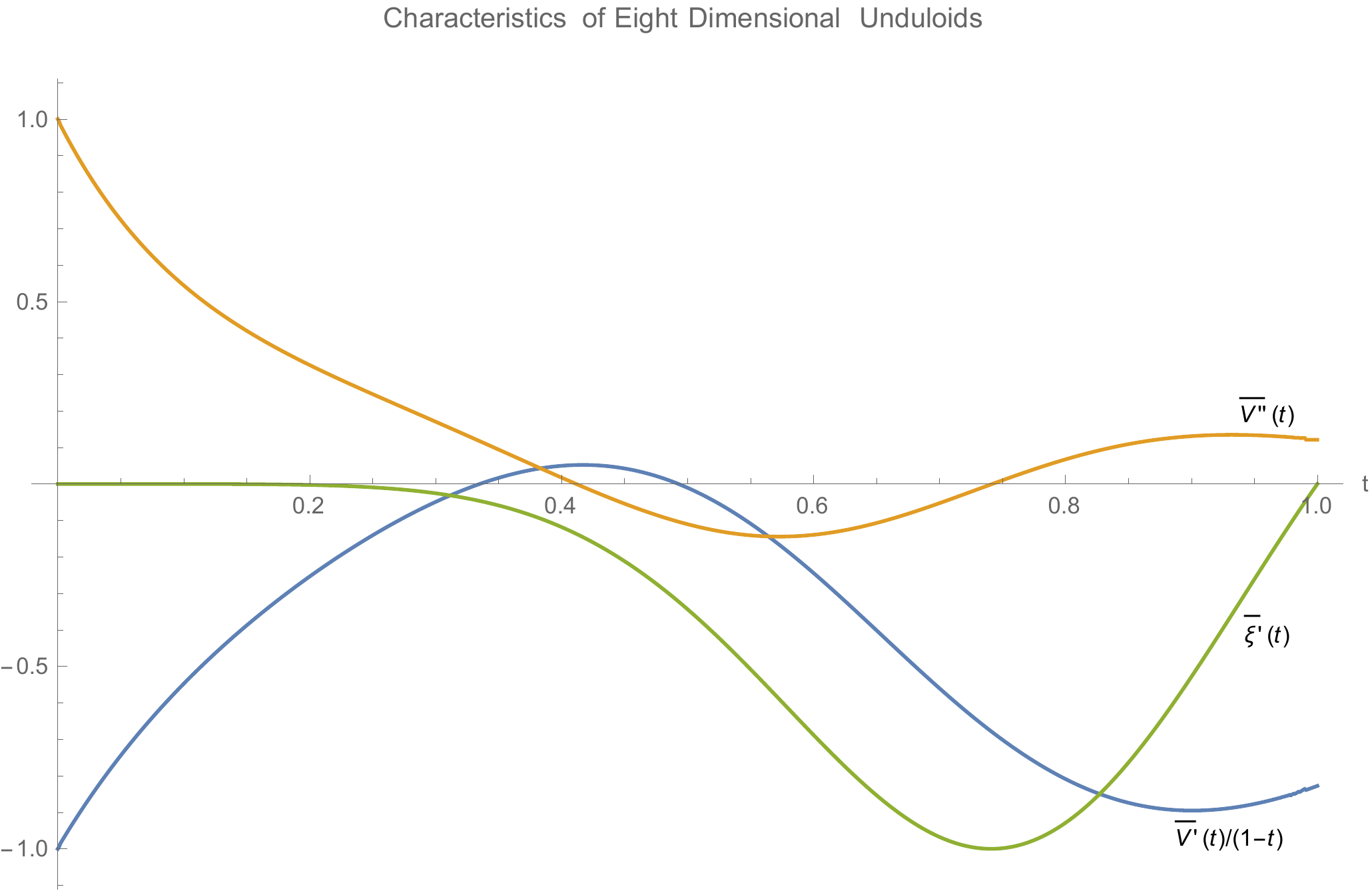}
\caption{Plots of dimension eight unduloid characteristics: $\frac{V'(t)}{1-t}$ (blue), $V''(t)$ (orange), and $\xi'(t)$ (green).}\label{Fig1}
\end{figure}

From this it is clear that at a point $t_0\in(0,1)$ where $V'(t_0)=0$ we have $\xi'(t_0)<0$ and $V''(t_0)\neq0$. Thus the conditions of Theorem \ref{stability} are met, since $\eta'(t_0)=\frac{\xi'(t_0)}{9\eta(t_0)^9V(t_0)}<0$. Further we see that there exists an open interval $I\subset(0,1)$ such that $V'(t)>0$ on $I$ and hence the unduloids $v(\cdot;t)$ for $t\in I$ are stable.
\end{proof}

\begin{remark}
We note that through a change of variables we can write the function $V(t)$ as:
\begin{align*}
V(t)=&\om_n\int_0^dv(z;t)^n\,dz\\
=&\frac{\om_n d^n}{P(t)^n}\int_0^d\left(1-(1-t)\zeta^{-1}\left(\frac{P(t)z}{d};t\right)\right)^n\,dz\\
=&\frac{\om_n d^n}{P(t)^n}\int_1^0\left(1-(1-t)x\right)^n\frac{-d}{P(t)}R(x;t)^{-1}\,dx\\
=&\frac{\om_n d^{n+1}}{P(t)^{n+1}}\int_0^1\left(1-(1-t)x\right)^nR(x;t)^{-1}\,dx.
\end{align*}
Also, note that $V(t^{-1})=V(t)$ and so $V'(1)=0$.
\end{remark}

\begin{remark}\label{MorganRem}
Despite these unduloids being stable, computations show that they do not in fact solve the isoperimetric problem. For a particular $t\in(0,1)$ the area of a cylinder with the same volume as $v(;,t)$ is given by $SA_c(t)=n\om_n^{\frac1n}V(t)^{\frac{n-1}n}d^{\frac1n}$ and the area of the half sphere of the same volume is $SA_s(t)=2^{\frac{-1}{n+1}}(n+1)\om_{n+1}^{\frac{1}{n+1}}V(t)^{\frac{n}{n+1}}$ (provided $\frac{2V(t)}{\om_{n+1}}\leq d^{n+1}$). By subtracting the area of the unduloid:
\begin{align*}
SA_u(t)=&n\om_n\int_0^dv(z;t)^{n-1}\sqrt{1+v_z(z;t)^2}\,dz\\
=&\frac{n\om_n d^{n-1}}{P(t)^{n-1}}\int_0^d\left(1-(1-t)\zeta^{-1}\left(\frac{P(t)z}{d};t\right)\right)^{n-1}\sqrt{1+(1-t)^2R\left(\zeta^{-1}\left(\frac{P(t)z}{d};t\right);t\right)^2}\,dz\\
=&\frac{n\om_nd^n}{P(t)^n}\int_0^1(1-(1-t)x)^{n-1}\sqrt{R(x;t)^{-2}+(1-t)^2}\,dx,
\end{align*}
from these areas we can plot the difference for dimension eight, see Figure \ref{Fig2}.
\begin{figure}[h]
\centering
\includegraphics[width=\textwidth]{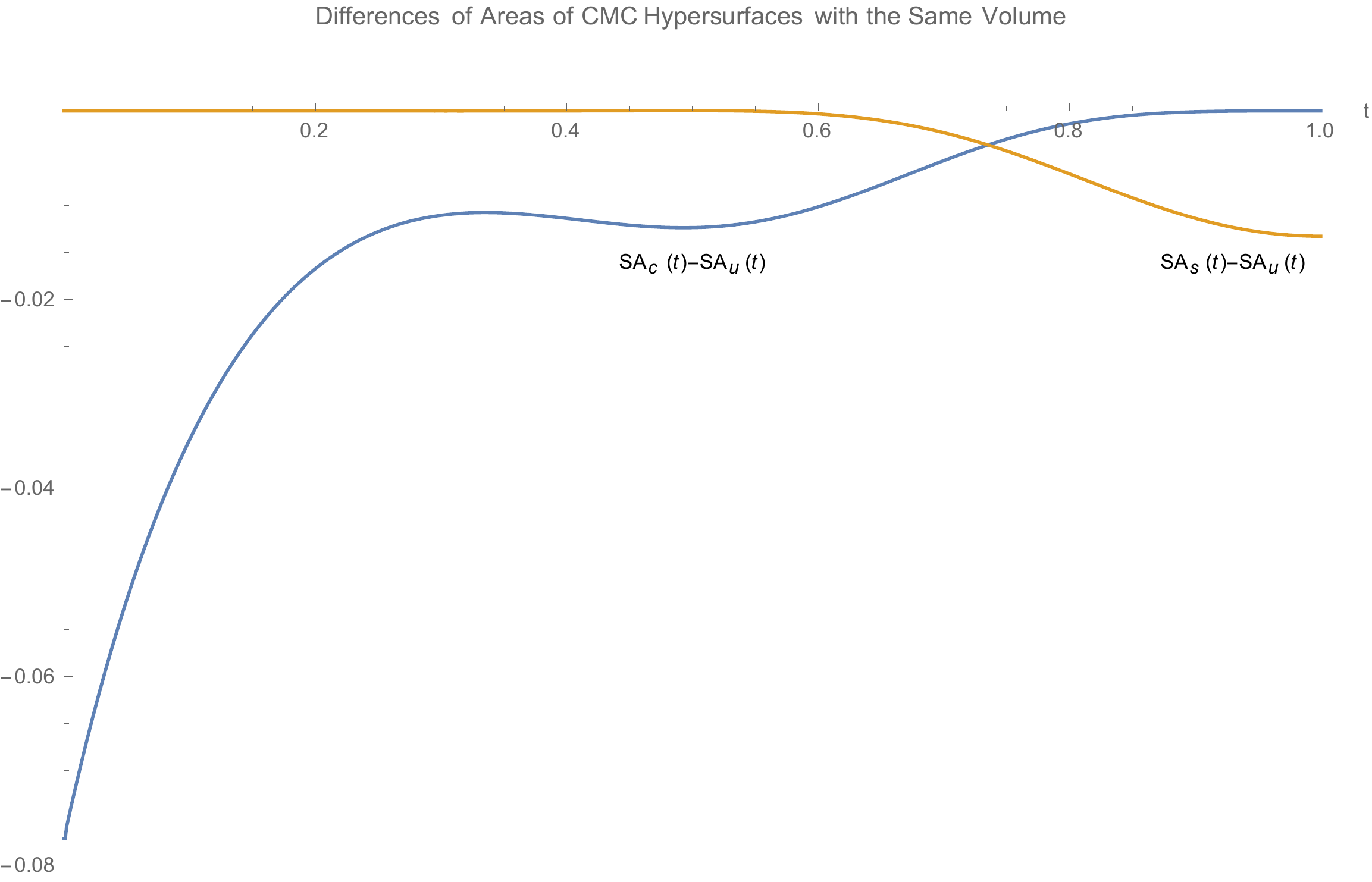}
\caption{Differences in area for a cylinder (blue) and half sphere (orange) compared to unduloids with the same volume in dimension eight with $d=1$.}\label{Fig2}
\end{figure}

The condition for a valid half sphere, $\frac{2V(t)}{\om_{n+1}}\leq d^{n+1}$, is always satisfied. This shows that any unduloid in dimension eight has a larger area than the corresponding cylinder and/or half sphere of the same volume.
\end{remark}


\section{The Stability Operator for CMC Hypersurfaces}\label{SecStable}
We will now recast the problem of stability in terms of an operator applicable to our situation. We only consider the case where $\Si$ is rotationally symmetric, around an axis in the $z$ direction, with free boundary and defined by a profile curve $v:[0,d]\to\RR_+$. With this assumption the functional (\ref{StabilityGen}) becomes
\[
J_v(u)=\frac{\int_{\SS^{n-1}}\int_0^d \left(\frac{1}{1+v_z^2}u_z^2+\frac{1}{v^2}\|\tilde{\nabla} u\|_{\SS^{n-1}}^2 -|A|^2u^2\right)v^{n-1}\sqrt{1+v_z^2}\,dz\,d\mu_{\SS^{n-1}}}{\int_{\SS^{n-1}}\int_0^d u^2v^{n-1}\sqrt{1+v_z^2}\,dz\,d\mu_{\SS^{n-1}}},
\]
for $u:[0,d]\times\SS^{n-1}\to\RR$ that satisfies $\int_{\SS^{n-1}}\int_0^d uv^{n-1}\sqrt{1+v_z^2}\,dz\,d\mu_{\SS^{n-1}}=0$ and the free boundary condition $u_z|_{z=0}=u_z|_{z=d}=0$.

We perform the function substitution $w=u\sqrt{1+v_z^2}$ and create the functional $\tilde{J}_v(w)=J_v\left(\frac{w}{\sqrt{1+v_z^2}}\right)$, for functions $w$ satisfying $\int_{\SS^{n-1}}\int_0^d wv^{n-1}\,dz\,d\mu_{\SS^{n-1}}=0$ and the free boundary condition. The functional is given by (see Lemma \ref{Jtil})
\[
\tilde{J}_v(w)=\frac{\int_{\SS^{n-1}}\int_0^d \left(\|\nabla w\|_{\Si}^2-\frac{(n-1)w^2}{v^2} \right)\frac{v^{n-1}}{\sqrt{1+v_z^2}}\,dz\,d\mu_{\SS^{n-1}}}{\int_{\SS^{n-1}}\int_0^d w^2\frac{v^{n-1}}{\sqrt{1+v_z^2}}\,dz\,d\mu_{\SS^{n-1}}}.
\]

Due to the symmetry of $v(\cdot;t)$ we only need to prove this is positive on rotationally symmetric functions, see the reasoning \cite{Vogel87} (for example). The functional acting on rotationally symmetric functions is:
\begin{align*}
\tilde{J}_v(w)=&\frac{\int_0^d\left(\frac{w_z^2}{(1+v_z^2)^{\frac32}}-\frac{(n-1)w^2}{v^2\sqrt{1+v_z^2}}\right)v^{n-1}\,dz}{\int_0^d  w^2\frac{v^{n-1}}{\sqrt{1+v_z^2}}\,dz}\\
=&\frac{\int_0^d\left(-\frac{w_{zz}}{(1+v_z^2)^{\frac32}}+\frac{3v_zv_{zz}w_z}{(1+v_z^2)^{\frac52}}-\frac{(n-1)v_zw_z}{v(1+v_z^2)^{\frac32}}-\frac{(n-1)w}{v^2\sqrt{1+v_z^2}}\right)wv^{n-1}\,dz}{\int_0^d  w^2\frac{v^{n-1}}{\sqrt{1+v_z^2}}\,dz}\\
=&\frac{\int_0^d DH(v)[w]wv^{n-1}\,dz}{\int_0^d  w^2\frac{v^{n-1}}{\sqrt{1+v_z^2}}\,dz},
\end{align*}
where $H$ is the mean curvature functional on rotationally symmetric functions:
\[
H(u)=\frac{-u_{zz}}{(1+u_z^2)^{\frac32}}+\frac{n-1}{u\sqrt{1+u_z^2}}.
\]
Since we are only considering functions, $w$, satisfying $\int_0^d wv^{n-1}\,dz=0$, this can also be written as
\[
\tilde{J}_v(w)=\frac{\int_0^d \left(DH(v)[w]-\frac{\int_0^d DH(v)[w]v^{n-1}\,dz}{\int_0^d v^{n-1}\,dz}\right)wv^{n-1}\,dz}{\int_0^d  w^2\frac{v^{n-1}}{\sqrt{1+v_z^2}}\,dz},
\]
so $\tilde{J}_v$ is positive on functions satisfying $\int_0^d wv^{n-1}\,dz=0$ and the free boundary condition if and only if
\[
A(v)[w]:=DH(v)[w]-\frac{\int_0^d DH(v)[w]v^{n-1}\,dz}{\int_0^d v^{n-1}\,dz}
\]
is positive definite on the same space of functions. However, since $A(v)$ maps back into functions with weighted mean zero, this is equivalent to proving all its eigenvalues are greater than zero.

\begin{proposition}\label{EquivStable}
A rotationally symmetric CMC hypersurface defined by the function $v$ is stable if the operator $A(v):X\to X$ only has positive eigenvalues, where $X:=\{w:[0,d]\to\RR:\int_0^dwv^{n-1}\,dz=0,\ w_z|_{z=0}=w_z|_{z=d}=0\}$.
\end{proposition}


\section{Null Space of the Stability Operator}\label{SecNull}
In order the determine the stability of unduloids in our family, it is necessary to determine the critical cases, that is at which $t$ does the operator $\tilde{A}(t):=A(v(\cdot;t))$ have a zero eigenvalue when acting on the space of functions:
\[
X_t:=\{w:[0,d]\to\RR:\int_0^dw(z)v(z;t)^{n-1}\,dz=0,\ w_z|_{z=0}=w_z|_{z=d}=0\}.
\]
A zero eigenvalue of $\tilde{A}(t)$ means that the linearised mean curvature operator is constant, so we first consider this situation without restricting to functions in $X_t$.

\begin{theorem}\label{ConsLinMC}
Let $w$ satisfy $DH(v(\cdot;t_0))[w]=C$ for some $t_0>0$ and $C\in\RR$. We have three cases:
\begin{itemize}
\item If $\eta'(t_0)\neq0$, there exists $\al,\be\in\RR$ such that
\[
w(z)=\al u_r(z;t_0,t_0)+\be u_z(z;t_0,t_0)+\frac{C}{\eta'(t_0)}u_t(z;t_0,t_0).
\]
\item If $\eta'(t_0)=0$ and $t_0\neq1$, let $k\geq2$ be the first $k$ such that $\eta^{(k)}(t_0)\neq0$, then there exists $\al,\be\in\RR$ such that
\[
w(z)=\al u_r(z;t_0,t_0)+\be u_z(z;t_0,t_0)+\frac{C}{\eta^{(k)}(t_0)}u_{t^{k}}(z;t_0,t_0).
\]
\item If $t_0=1$, there exists $\al,\be\in\RR$ such that
\[
w(z)=\al \cos\left(\frac{\pi z}{d}\right)+\be \sin\left(\frac{\pi z}{d}\right) -\frac{Cd^2}{\pi^2}.
\]
\end{itemize}
\end{theorem}

\begin{proof}
We start by considering the $\eta'(t_0)\neq0$ case. By differentiating the equation $H(u(z;r,t))=\eta(t)$ with respect to each of the variables we obtain
\begin{equation}\label{HuDerivs}
\begin{array}{cc}
DH(u(z;r,t))[u_z(z;r,t)]=0, &  DH(u(z;r,t))[u_r(z;r,t)]=0,\\
\multicolumn{2}{c}{DH(u(z;r,t))[u_t(z;r,t)]=\eta'(t),}
\end{array}
\end{equation}
so at $r=t=t_0$ we have
\begin{equation}\label{HuDerivs0}
\begin{array}{cc}
DH(v(z;t_0))[u_z(z;t_0,t_0)]=0, & DH(v(z;t_0))[u_r(z;t_0,t_0)]=0,\\
\multicolumn{2}{c}{DH(v(z;t_0))[u_t(z;t_0,t_0)]=\eta'(t_0),}
\end{array}
\end{equation}
so the statement follows from standard linear second order DE theory provided $u_z(z;t_0,t_0)$ and $u_r(z;t_0,t_0)$ are linearly independent. To see this is the case we note that $u_z(0;t_0,t_0)=0$ and $u_z(z;t_0,t_0)\equiv0$ if and only if $t_0=1$, which is excluded from this case. Therefore the functions are linearly independent if $u_r(0;t_0,t_0)\neq0$. We have $u(z;r,t)=\frac{nQ(r)}{\eta(t)}\left(1-(1-r)\zeta^{-1}\left(\frac{\eta(t)z}{nQ(r)};r\right)\right)$, so
\begin{align*}
u_r(z;r,t)=&\frac{nQ'(r)}{\eta(t)}\left(1-(1-r)\zeta^{-1}\left(\frac{\eta(t)z}{nQ(r)};r\right)\right)+\frac{nQ(r)}{\eta(t)}\zeta^{-1}\left(\frac{\eta(t)z}{nQ(r)};r\right)\\
&+\frac{(1-r)Q'(r)z}{Q(r)}\zeta^{-1}_y\left(\frac{\eta(t)z}{nQ(r)};r\right)-\frac{n(1-r)Q(r)}{\eta(t)}\zeta^{-1}_r\left(\frac{\eta(t)z}{nQ(r)};r\right),
\end{align*}
and hence by using $\zeta^{-1}(0;r)=1$ (and therefore $\zeta^{-1}(0;r)=0$)
\[
u_r(0;r,t)=\frac{n(rQ'(r)+Q(r))}{\eta(t)}=\frac{n(rQ(r))'}{\eta(t)}>0.
\]

Next we consider the case where $\eta'(t_0)=0$ and $t_0\neq1$. We have
\begin{align*}
u_t(z;r,t)=&-\frac{nQ(r)\eta'(t)}{\eta(t)^2}\left(1-(1-r)\zeta^{-1}\left(\frac{\eta(t)z}{nQ(r)};r\right)\right)-\frac{(1-r)\eta'(t)z}{\eta(t)}\zeta^{-1}_y\left(\frac{\eta(t)z}{nQ(r)};r\right)\\
=&-\frac{\eta'(t)}{\eta(t)}\left(\frac{nQ(r)}{\eta(t)}\left(1-(1-r)\zeta^{-1}\left(\frac{\eta(t)z}{nQ(r)};r\right)\right)+(1-r)z\zeta^{-1}_y\left(\frac{\eta(t)z}{nQ(r)};r\right)\right),
\end{align*}
and hence $u_t(z;t_0,t_0)\equiv0$, in fact $u_{t^i}(z;t_0,t_0)\equiv0$ for all $i=1,\ldots,k-1$. By now differentiating the final equation of (\ref{HuDerivs}) again with respect to $t$ we obtain
\[
D^2H(u(z;r,t))[u_t(z;r,t),u_t(z;r,t)]+DH(u(z;r,t))[u_{tt}(z;r,t)]=\eta''(t)
\]
Therefore $DH(v(z;t_0))[u_{tt}(z;t_0,t_0)]=\eta''(t_0)$ and it is clear that
\[
DH(v(z;t_0))[u_{t^i}(z;t_0,t_0)]=\eta^{(i)}(t_0),
\]
for all $i=1,\ldots,k$. In particular, $DH(v(z;t_0))[u_{t^k}(z;t_0,t_0)]=\eta^{(k)}(t_0)\neq0$. The proof of linear independence of $u_z(z;t_0,t_0)$ and $u_r(z;t_0,t_0)$ is the same as in the previous case, so we obtain the conclusion.

Lastly, we consider the case when $t_0=1$. In this case we use that $v(z;1)=\frac{d}{P(1)}$ to write out the differential equation for $w$:
\[
DH\left(\frac{d}{P(1)}\right)[w]=-w''(z)-\frac{(n-1)P(1)^2}{d^2}w(z)=C,
\]
which, using that $P(1)=\int_0^1\frac{1}{\sqrt{(n-1)(1-x)x}}\,dx=\frac{\pi}{\sqrt{n-1}}$, is easily solved to give the conclusion.
\end{proof}

\begin{remark}
We note that in the $t_0=1$ case, $\zeta(x;1)=\frac{2}{\sqrt{n-1}}\arcsin(\sqrt{1-x})$ and hence $\zeta^{-1}(y;1)=\cos^2\left(\frac{\sqrt{n-1}y}{2}\right)=\frac{1}{2}+\frac{1}{2}\cos(\sqrt{n-1}y)$. Therefore
\[
u_r(z;1,1)=\frac{nQ'(1)}{\eta(1)}+\frac{nQ(1)}{\eta(1)}\zeta^{-1}\left(\frac{\eta(1)z}{nQ(1)};1\right)=\frac{d\sqrt{n-1}}{2\pi}\cos\left(\frac{\pi z}{d}\right).
\]
Also $u_{tt}(z;1,1)\equiv-\frac{d^2\eta''(1)}{\pi^2}=-\frac{(n^2-10n+10)\sqrt{n-1}d}{48\pi}\neq0$, so the second case holds with $k=2$ and $u_z(z;t_0,t_0)$ replaced with $\sin\left(\frac{\pi z}{d}\right)$.
\end{remark}

\begin{corollary}\label{NullA}
If $V'(t_0)\neq0$ then $\textrm{Null}(\tilde{A}(t_0))$ is empty, otherwise it is one dimensional and $\textrm{Null}(\tilde{A}(t_0))=\textrm{span}\{v_t(z;t_0)\}$.
\end{corollary}

\begin{proof}
We consider the three cases found in Theorem \ref{ConsLinMC} separately and determine when there is a non-trivial solution in $X_t$.

When $\eta'(t_0)\neq0$ we have $w(z)=\al u_r(z;t_0,t_0)+\be u_z(z;t_0,t_0)+\frac{C}{\eta'(t_0)}u_t(z;t_0,t_0)$ for some $\al,\be,C\in\RR$ and so
\[
w'(z)=\al u_{rz}(z;t_0,t_0)+\be u_{zz}(z;t_0,t_0)+\frac{C}{\eta'(t_0)}u_{tz}(z;t_0,t_0).
\]
Using that $u_z(0;r,t)=0$ for any $r,t>0$ we therefore have $w'(0)=\be u_{zz}(0;t_0,t_0)$, and $u_{zz}(0;t_0,t_0)=\left(\frac{(n-1)P(t_0)}{dt_0}-\eta(t_0)\right)=\eta(t_0)\left(\frac{n-1}{nt_0Q(t_0)}-1\right)$ using the CMC equation. So $u_{zz}(0;t_0,t_0)=0$ if and only if $t_0=1$ (using that $tQ(t)$ is increasing), which is excluded from this case, and hence we require $\be=0$. Next we use that since $u_z\left(\frac{\eta(r)d}{\eta(t)};r,t\right)=0$ we have
\begin{equation}\label{Dud1}
u_{zz}\left(\frac{\eta(r)d}{\eta(t)};r,t\right)\frac{\eta'(r)d}{\eta(t)}+u_{zr}\left(\frac{\eta(r)d}{\eta(t)};r,t\right)=0,
\end{equation}
and
\begin{equation}\label{Dud2}
-u_{zz}\left(\frac{\eta(r)d}{\eta(t)};r,t\right)\frac{\eta(r)\eta'(t)d}{\eta(t)^2}+u_{zt}\left(\frac{\eta(r)d}{\eta(t)};r,t\right)=0,
\end{equation}
to conclude that
\[
w'(d)=-\frac{\al\eta'(t_0)d}{\eta(t_0)}u_{zz}(d;t_0,t_0)+\frac{Cd}{\eta(t_0)}u_{zz}(d;t_0,t_0)=\frac{d(C-\al\eta'(t_0))}{\eta'(t_0)}u_{zz}(d;t_0,t_0).
\]
We use that $u_{zz}(d;t_0,t_0)=\frac{(n-1)P(t_0)}{d}-\eta(t_0)=\eta(t_0)\left(\frac{n-1}{nQ(t_0)}-1\right)$, which again is zero only in the excluded $t_0=1$ case (since $Q(t)$ is decreasing), to conclude that $\al=\frac{C}{\eta'(t_0)}$. Hence $w(z)=\frac{C}{\eta'(t_0)}\left(u_{r}(z;t_0,t_0)+u_{t}(z;t_0,t_0)\right)=\frac{C}{\eta'(t_0)}v_{t}(z;t_0)$ with $C\neq0$. We now calculate the weighted integral:
\[
\int_0^d\frac{C}{\eta'(t_0)}v_t(z;t_0)v(z;t_0)^{n-1}\,dz=\frac{C}{n\om_n\eta'(t_0)}V'(t_0),
\]
and hence we require $V'(t_0)=0$ for the null space to be non-empty, in which case a spanning function is $v_{t}(z;t_0)$.

Next we consider when $\eta'(t_0)=0$ and $t_0\neq1$, so
\[
w(z)=\al u_r(z;t_0,t_0)+\be u_z(z;t_0,t_0)+\frac{C}{\eta^{(k)}(t_0)}u_{t^k}(z;t_0,t_0).
\]
As above we use that $w'(z)=\al u_{rz}(z;t_0,t_0)+\be u_{zz}(z;t_0,t_0)+\frac{C}{\eta^{(k)}(t_0)}u_{t^kz}(z;t_0,t_0)$ to obtain $w'(0)=\be u_{zz}(0;t_0,t_0)$ and conclude that $\be=0$. By differentiating (\ref{Dud2}) another $k-1$ times and using that $\eta^{(i)}(t_0)=0$ for all $i=1,\ldots,k-1$ we obtain
\[
-u_{zz}(d;t_0,t_0)\frac{\eta^{(k)}(t_0)d}{\eta(t_0)}+u_{zt^k}(d;t_0,t_0)=0,
\]
and hence by also using (\ref{Dud1}) we obtain
\[
w'(d)=-\frac{\al\eta'(t_0)d}{\eta(t_0)}u_{zz}(d;t_0,t_0)+\frac{Cd}{\eta(t_0)}u_{zz}(d;t_0,t_0)=\frac{Cd}{\eta(t_0)}u_{zz}(d;t_0,t_0).
\]
Therefore $C=0$, in which case $w(z)=\al u_r(z;t_0,t_0)$ with $\al\neq0$. However, since $u_t(z;t_0,t_0)\equiv0$ we can also write this as $w(z)=\al v_t(z;t_0)$ and obtain $\int_0^dw(z)v(z;t_0)^{n-1}\,dz=\frac{\al}{n\om_n} V'(t_0)$, so we again require $V'(t_0)=0$ for the null space to be non-empty, in which case the spanning function is $v_{t}(z;t_0)$.

Finally, we consider the case $t_0=1$. In this case
\[
w(z)=\al \cos\left(\frac{\pi z}{d}\right)+\be \sin\left(\frac{\pi z}{d}\right) -\frac{Cd^2}{\pi^2}.
\]
From $w'(z)=\frac{\pi}{d}\left(-\al \sin\left(\frac{\pi z}{d}\right)+\be\cos\left(\frac{\pi z}{d}\right)\right)$, we again see that $w'(0)=0$ if and only if $\be=0$, but now $w'(d)=0$ gives no further condition. We consider the weighted integral, using that $v(z;1)=\frac{d\sqrt{n-1}}{\pi}$:
\begin{align*}
\int_0^dw(z)v(z;1)^{n-1}\,dz=&\left(\frac{d\sqrt{n-1}}{\pi}\right)^{n-1}\int_0^d\al\cos\left(\frac{\pi z}{d}\right)-\frac{Cd^2}{\pi^2}\,dz\\
=&-\frac{Cd^{n+2}(n-1)^{\frac{n-1}2}}{\pi^{n+1}},
\end{align*}
and hence we require $C=0$. Therefore, the null space is one dimensional with spanning function $\cos\left(\frac{\pi z}{d}\right)$, since $V'(1)=0$ and $\cos\left(\frac{\pi z}{d}\right)=\frac{2\pi\al}{d\sqrt{n-1}}v_t(z;1)$ we complete the theorem.
\end{proof}

\begin{remark}
From the formula:
\[
\frac{\int A(v)[w]wv^{n-1}\,dz}{\int  w^2\frac{v^{n-1}}{\sqrt{1+v_z^2}}\,dz}=J_v(w)=\frac{\int\left(\frac{w_z^2}{(1+v_z^2)^{\frac32}}-\frac{(n-1)w^2}{v^2\sqrt{1+v_z^2}}\right)v^{n-1}\,dz}{\int  w^2\frac{v^{n-1}}{\sqrt{1+v_z^2}}\,dz},
\]
it is easily seen that all eigenvalues of $\tilde{A}(t)$ are real and form a sequence going to infinity. It also shows that they are greater than or equal to $\frac{-(n-1)}{\min_{z\in[0,d]}v(\cdot;t)^2}$. This lower bound and Chapter 4 Theorem 3.16 in Kato \cite{Kato76} ensures that for any eigenvalue $\lambda_j(t)$ that changes sign, it must do so by moving through $0$.
\end{remark}


\section{Variation of a Zero Eigenvalue}\label{SecVary}
In this section we determine how an eigenvalue of $\tilde{A}$ that becomes zero, varies in a neighbourhood of its zero value. To do this we need to alter the operators so that they are all defined on the same domain. To this end we introduce the projections
\[
W(t)[u]:=u(z)-\frac{\int_0^du(z)v(z;t)^{n-1}\,dz}{\int_0^dv(z;t)^{n-1}\,dz},
\]
which are bijections between the projected spaces $X_t$, that is $W(t):X_r\to X_t$ is a bijection for any $r,t>0$. Also note that $W(r)\circ W(t)=W(r)$. Now we define the family of operators $B(t):X_1\to X_1$ given by:
\[
B(t)=W(1)\circ \tilde{A}(t)\circ W(t)=W(1)\circ DH(v(\cdot;t))\circ W(t),
\]
and note that since $\tilde{A}(t)=W(t)\circ DH(v(\cdot;t))$, we have $\tilde{A}(t)=W(t)\circ B(t)\circ W(1)$.
\begin{lemma}\label{SpecB}
$B(t)$ has the same spectrum as $\tilde{A}(t)$.
\end{lemma}

\begin{proof}
This follows directly from the formulas above. If $u\in X_t$ is an eigenfunction of $\tilde{A}(t)$, then $W(1)[u]\in X_1\backslash\{0\}$ is an eigenfunction of $B(t)$ with the same eigenvalue and similarly if $u\in X_1$ is an eigenfunction of $B(t)$, then $W(t)[u]\in X_t\backslash\{0\}$ is an eigenfunction of $\tilde{A}(t)$ with the same eigenvalue.
\end{proof}

By Corollary \ref{NullA} any zero eigenvalue of $B(t)$ has multiplicity $1$. To see that it is in fact simple we ensure that a spanning vector, $W(1)[v_t(\cdot;t)]$, is not in the range. This is not quite trivial since $B(t)$ is not self adjoint. However, a solution $\bar{w}\in X_1$ to the differential equation $W(1)[DH(v(z;t_0))[W(t_0)[\bar{w}]]]=W(1)[v_t(z;t_0)]$ exists if and only if $W(t_0)[DH(v(z;t_0))[W(t_0)[\bar{w}]]]=v_t(z;t_0)$ (using that $v_t(\cdot;t_0)\in X_{t_0}$). The operator $W(t_0)\circ DH(v(\cdot;t_0))\circ W(t_0)$ is self adjoint on $X_{t_0}$ and $v_t(\cdot;t_0)$ is an element of its null space, so the differential equations have no solutions. Therefore, any zero eigenvalue of $B(t)$ is simple.

To prove the existence of continuously differentiable eigenvalues and eigenfunctions we will use Proposition I.7.2 of \cite{Kielhofer12}. For this we need to find a function for which $B(t)$ is the linearisation. We define $\bar{v}(\cdot;t)=W(1)[v(\cdot;t)]$ and start with a lemma that we prove in Section \ref{PsiSec}.

\begin{lemma}\label{PsiExist}
For each $t>0$ there exist open neighbourhoods of $\bar{v}(\cdot;t)$, $U_{t}\subset X_1$, and $t$, $V_{t}\subset\RR$, and a function $\psi_{t}:U_{t}\times V_{t}\to X_t$ such that
\begin{itemize}
\item $W(1)[\psi_{t}(\bar{u},r)]=\bar{u}$ for all $(\bar{u},r)\in U_{t}\times V_{t}$,
\item $\om_n\int_0^d\psi_{t}(\bar{u},r)^n\,dz=V(r)$ for all $(\bar{u},r)\in U_{t}\times V_{t}$, and
\item $W(1)[u]=\bar{u}$ and $\int_0^du(z)^n\,dz=V(r)$ for $(u,\bar{u},r)\in R(\psi_{t})\times U_{t}\times V_{t}$ if and only if $u=\psi_{t}(\bar{u},r)$.
\end{itemize}
In particular, there exists a neighbourhood of $t$, $I_t\subset\RR$, such that $\psi_{t}(\bar{v}(\cdot;r),r)=v(\cdot;r)$ for all $r\in I_t$.

Furthermore,
\[
D_{\bar{u}}\psi_{t}(\bar{u},r)[\bar{w}]=\bar{w}-\frac{\int_0^d\bar{w}(z)\psi_{t}(\bar{u},r)(z)^{n-1}\,dz}{\int_0^d\psi_{t}(\bar{u},r)(z)^n\,dz},
\]
for all $(\bar{u},r)\in U_{t}\times V_{t}$ and $\bar{w}\in X_1$. In particular, $D_{\bar{u}}\psi_{t}(\bar{v}(\cdot;r),r)[\bar{w}]=W(r)[\bar{w}]$ for all $\bar{w}\in X_1$ and $r\in I_t$.
\end{lemma}

We now define the function $F_t:U_t\times V_t\to X_1$ by $F_{r,t}(\bar{u},r)=W(1)[H(\psi_{t}(\bar{u},r))]$. It is easily seen that $F_t(\bar{v}(\cdot;r),r)=0$ and $D_{\bar{u}}F_t(\bar{v}(\cdot;r),r)[\bar{w}]=B(t)[\bar{w}]$. Therefore by Proposition I.7.2 of \cite{Kielhofer12} the zero eigenvalue will vary smoothly with $r$ in a neighbourhood of $t$. That is, let $t_0>0$ be such that $V'(t_0)=0$ then there exists an open neighbourhood of $t_0$, $J_{t_0}\subset I_{t_0}$, and a smoothly varying family of functions $w:X_1\times J_{t_0}\to\RR$, such that $\la\in C^{\infty}(J_{t_0},\RR)$, $\la(t_0)=0$, $w(\cdot;t_0)=0$ and
\[
B(t)[\bar{v}_t(z;t_0)+w(z;t)]=\la(t)(\bar{v}_t(z;t_0)+w(z;t)).
\]
We can also take $\int_0^d\bar{v}_t(z;t_0)w(z;t)\,dt=0$ for all $t\in J_{t_0}$.

\begin{lemma}\label{Dla}
\[
\la'(t_0)=\frac{-V''(t_0)\eta'(t_0)}{n\om_n\int_0^dv_t(z;t_0)^2v(z;t_0)^{n-1}\,dz}
\]
\end{lemma}

\begin{proof}
We replace $B(t)$ with $W(1)\circ DH(v(\cdot;t))\circ W(t)$ and use that $W(t)[\bar{v}_t(z;t_0)]=W(t)[v_t(z;t_0)]$:
\[
W(1)[DH(v(z;t))[W(t)[v_t(z;t_0)+w(z;t)]]]=\la(t)(\bar{v}_t(z;t_0)+w(z;t)).
\]
Calculating its $t$ derivative at $t=t_0$ gives:
\begin{align*}
W&(1)[D^2H(v(z;t_0))[v_t(z;t_0),v_t(z;t_0)]]\\
&+W(1)[DH(v(z;t_0))[W'(t_0)[v_t(z;t_0)]+W(t_0)[w_t(z;t_0)]]]=\la'(t_0)\bar{v}_t(z;t_0).
\end{align*}
We apply $W(t_0)$ to the equation to obtain
\begin{align*}
W&(t_0)[D^2H(v(z;t_0))[v_t(z;t_0),v_t(z;t_0)]]\\
&+W(t_0)[DH(v(z;t_0))[W'(t_0)[v_t(z;t_0)]+W(t_0)[w_t(z;t_0)]]]=\la'(t_0)v_t(z;t_0).
\end{align*}

Next we use that since $H(v(z;t))=\eta(t)$, then $DH(v(z;t))[v_t(z;t)]=\eta'(t)$ and $D^2H(v(z;t)[v_t(z;t),v_t(z;t)]+DH(v(z;t))[v_{tt}(z,t)]=\eta''(t)$, to obtain
\begin{equation}\label{EvalueEqD}
W(t_0)[DH(v(z;t_0))[W'(t_0)[v_t(z;t_0)]+W(t_0)[w_t(z;t_0)]-v_{tt}(z;t_0)]]=\la'(t_0)v_t(z;t_0),
\end{equation}
where we used that $W(t_0)[\eta''(t_0)]=0$.

We now consider $V(t)=\om_n\int_0^dv(z;t)^n\,dz$. From this we see that
\[
V'(t)=n\om_n(I-W(t))[v_t(z;t)]\int_0^dv(z;t)^{n-1}\,dz,
\]
and
\begin{align*}
V''(t)=&n\om_n\left(\left(I-W(t)\right)[v_{tt}(z;t)]-W'(t)[v_t(z;t)]\right)\int_0^dv(z;t)^{n-1}\,dz\\
&+n(n-1)\om_n(I-W(t))[v_t(z;t)]\int_0^dv_t(z;t)v(z;t)^{n-2}\,dz.
\end{align*}
Hence
\[
V''(t_0)=n\om_n\left(\left(I-W(t_0)\right)[v_{tt}(z;t_0)]-W'(t_0)[v_t(z;t_0)]\right)\int_0^dv(z;t_0)^{n-1}\,dz,
\]
and
\[
W'(t_0)[v_t(z;t_0)]=(I-W(t_0))[v_{tt}(z;t_0)]-\frac{V''(t_0)}{n\om_n\int_0^dv(z;t_0)^{n-1}\,dz}.
\]
Substituting this into (\ref{EvalueEqD}) gives
\[
W(t_0)[DH(v(z;t_0))[W(t_0)[w_t(z;t_0)-v_{tt}(z;t_0)-\frac{V''(t_0)}{n\om_n\int_0^dv(z;t_0)^{n-1}\,dz}]]]=\la'(t_0)v_t(z;t_0).
\]
Multiplying by $v_t(z;t_0)v(z;t_0)^{n-1}$ and integrating gives:
\begin{align*}
\la'(t_0)=&\frac{\int_0^dDH(v(z;t_0))[W(t_0)[w_t(z;t_0)-v_{tt}(z;t_0)]]v_t(z;t_0)v(z;t_0)^{n-1}\,dz}{\int_0^dv_t(z;t_0)^2v(z;t_0)^{n-1}\,dz}\\
&-\frac{V''(t_0)\int_0^dDH(v(z;t_0))[1]v_t(z;t_0)v(z;t_0)^{n-1}\,dz}{n\om_n\int_0^dv(z;t_0)^{n-1}\,dz\int_0^dv_t(z;t_0)^2v(z;t_0)^{n-1}\,dz},
\end{align*}
where the projections vanish due to $v_t(z;t_0)\in X_{t_0}$. Using the self adjointness of $DH(v(\cdot;t))$ (with respect to the weight $v(\cdot;t_0)$, for functions satisfying the free boundary condition) and
\[
DH(v(z;t_0))[v_t(z;t_0)]=\eta'(t_0),
\]
this becomes:
\begin{align*}
\la'(t_0)=&\frac{\eta'(t_0)\int_0^dW(t_0)[w_t(z;t_0)-v_{tt}(z;t_0)]v(z;t_0)^{n-1}\,dz}{\int_0^dv_t(z;t_0)^2v(z;t_0)^{n-1}\,dz}\\
&-\frac{\eta'(t_0)V''(t_0)\int_0^dv(z;t_0)^{n-1}\,dz}{n\om_n\int_0^dv(z;t_0)^{n-1}\,dz\int_0^dv_t(z;t_0)^2v(z;t_0)^{n-1}\,dz},
\end{align*}
and the result follows since $W(t_0)$ projects into $X_{t_0}$.
\end{proof}

We can now prove the main theorem.
\proofofmaintheorem
We start by labeling the critical points of $V$ as $t_0=1$, $t_1$, $t_2$, etc. in decreasing order, i.e. $V'(t_i)=0$ and $0<t_{i+1}<t_i$ for all $i$.

Suppose at some $\tau\in(t_{i+1},t_{i})$ the eigenvalues of $\tilde{A}(\tau)$ are all positive and $V'(\tau)>0$. Both these things must remain true for $\tilde{A}(t)$ for all $t\in(t_{i+1},t_i)$. Therefore immediately after the critical point at $t_{i+1}$, $V'(t)>0$ and, since we assume no critical points of $V$ are degenerate, $t_{i+1}$ must be a minimum of $V$ and hence $V''(t_{i+1})>0$. From Lemma \ref{Dla} and since we assume $\eta'(t_0)<0$ this means that the critical eigenvalue (with multiplicity one) is increasing at this point and hence is negative immediately prior to $t_{i+1}$. Thus for $t\in(t_{i+2},t_{i+1})$ we have that $V'(t)<0$ and $\tilde{A}(t)$ has a single negative eigenvalue.

Next suppose at some $\tau\in(t_{i+1},t_{i})$, $\tilde{A}(\tau)$ has a single negative eigenvalue, the rest are strictly positive, and $V'(\tau)<0$. These things must remain true for $\tilde{A}(t)$ for all $t\in(t_{i+1},t_i)$. Therefore immediately after the critical point $t_{i+1}$, $V'(t)<0$ and, since we assume no critical points of $V$ are degenerate, $t_{i+1}$ must be a maximum of $V$ and hence $V''(t_{i+1})<0$. From Lemma \ref{Dla} and since we assume $\eta'(t_0)<0$ this means that the critical eigenvalue (with multiplicity one) is decreasing at this point and hence is positive immediately prior to $t_{i+1}$. Thus for $t\in(t_{i+2},t_{i+1})$ we have that $V'(t)>0$ and $\tilde{A}(t)$ has only positive eigenvalues.

As these two cases alternate, we have covered all cases if we can show that on some interval $(a,1)$ one of them is true. However, this was proved in \cite{Hartley16}. In fact, there it was shown that for $2\leq n\leq10$, $V(t)$ has a local minimum and the critical eigenvalue has a local maximum at $t=1$, so that there exists $\ep>0$ such that for $t\in(1-\ep,1)$, $\tilde{A}(t)$ has a single negative eigenvalue and $V'(t)<0$. While for $n\geq11$, $V(t)$ has a local maximum and the critical eigenvalue has a local minimums at $t=1$, so that there exists $\ep>0$ such that for $t\in(1-\ep,1)$, $\tilde{A}(t)$ has only positive eigenvalues and $V'(t)>0$. By Proposition \ref{EquivStable}, stability of the hypersurface defined by $v(\cdot;t)$ follows from $\tilde{A}(t)$ having only positive eigenvalues and it is also clear that if $\tilde{A}(t)$ has negative eigenvalues then the hypersurface is unstable. Hence we have proved the result. \hfill $\Box$

\begin{remark}
In \cite{Hartley16} the details are given for the operator which is the negative of $\tilde{A}(t)$, hence the eigenvalue signs are switched. Also, instead of using $V(t)$, the eigenvalue formulas are written in terms of the function $\frac{C}{V(t)^{\frac{1}{n}}}$. Finally, the family of height functions is parameterised by the different parameter $s=\frac{1-t}{1+t}$.
\end{remark}


\section{Mean Curvature Calculation for the Family of Hypersurfaces}\label{SecFamilyApp}
Here we prove Proposition \ref{FamCMC}. When $r=1$, $u(z;1,t)=\frac{(n-1)d}{nP(t)Q(t)}$ and the proposition follows trivially. We now assume $r\neq1$ and note that
\[
\zeta^{-1}_y(y;t)=\frac{1}{\zeta_x(\zeta^{-1}(y;t);t)}=-R(\zeta^{-1}(y;t);t).
\]
Therefore
\[
u_z(z;r,t)=(1-r)R\left(\zeta^{-1}\left(\frac{Q(t)P(t)z}{Q(r)d};r\right);r\right).
\]
Using $\zeta^{-1}(0;r)=1$ and $\zeta^{-1}(P(r);r)=0$ we obtain $u_z(0;r,t)=(1-r)R(1;r)=0$ and $u_z\left(\frac{Q(r)P(r)d}{Q(t)P(t)};r,t\right)=(1-r)R(0;r)=0$.

Next define $S(x;r):=\frac{(1-(1-r)x)^{n-1}}{1-Q(r)+Q(r)(1-(1-r)x)^n}$, and take its $x$-derivative:
\[
S_x(x;r)=-(1-r)\left(\frac{n-1}{1-(1-r)x}-nQ(r)S(x;r)\right)S(x;r).
\]
Therefore, from $R(x;r)=\frac{1}{|1-r|}\sqrt{S(x;r)^2-1}$ we have:
\begin{align*}
R_x(x;r)=&\frac{S(x;r)S_x(x;r)}{|1-r|\sqrt{S(x;r)^2-1}}\\
=&\frac{-S(x;r)^2}{(1-r)R(x;r)}\left(\frac{n-1}{1-(1-r)x}-nQ(r)S(x;r)\right)\\
=&-\frac{(1-r)^2R(x;r)^2+1}{(1-r)R(x;r)}\left(\frac{n-1}{1-(1-r)x}-nQ(r)\sqrt{(1-r)^2R(x;r)^2+1}\right).
\end{align*}
So that, with $x=\zeta^{-1}\left(\frac{Q(t)P(t)z}{Q(r)d};r\right)$, we have
\begin{align*}
u_{zz}(z;r,t)=&\frac{Q(t)P(t)\left((1-r)^2R(x;r)^2+1\right)}{Q(r)d}\left(\frac{n-1}{1-(1-r)x}-nQ(r)\sqrt{(1-r)^2R(x;r)^2+1}\right)\\
=&\frac{Q(t)P(t)\left(u_z(z;r,t)^2+1\right)}{Q(r)d}\left(\frac{n-1}{\frac{Q(t)P(t)}{Q(r)d}u(z;r,t)}-nQ(r)\sqrt{u_z(z;r,t)^2+1}\right)\\
=&\left(u_z(z;r,t)^2+1\right)\left(\frac{n-1}{u(z;r,t)}-\frac{nQ(t)P(t)}{d}\sqrt{u_z(z;r,t)^2+1}\right).
\end{align*}
The result now follows from the formula for $H(u)$ for rotationally symmetric functions.


\section{Calculation of the Stability Functional}\label{SecStableApp}
\begin{lemma}\label{Jtil}
\[
\tilde{J}_v(w)=\frac{\int_{\SS^{n-1}}\int_0^d \left(\|\nabla w\|_{\Si}^2-\frac{(n-1)w^2}{v^2} \right)\frac{v^{n-1}}{\sqrt{1+v_z^2}}\,dz\,d\mu_{\SS^{n-1}}}{\int_{\SS^{n-1}}\int_0^d w^2\frac{v^{n-1}}{\sqrt{1+v_z^2}}\,dz\,d\mu_{\SS^{n-1}}}
\]
\end{lemma}

\begin{proof}
To perform the substitution $u=\frac{w}{\sqrt{1+v_z^2}}$ in $J_v(u)$, we start by calculating the $\frac{u_z^2u^{n-1}}{\sqrt{1+v_z^2}}$ term:
\begin{align*}
\frac{u_z^2u^{n-1}}{\sqrt{1+v_z^2}}=&\left(\frac{w_z^2}{1+v_z}-\frac{2v_zv_{zz}ww_z}{(1+v_z^2)^2}+\frac{v_z^2v_{zz}^2w^2}{(1+v_z^2)^3}\right)\frac{v^{n-1}}{\sqrt{1+v_z^2}}\\
=&\frac{\pa}{\pa z}\left(\frac{-v^{n-1}v_zv_{zz}w^2}{(1+v_z^2)^{\frac52}}\right)+\left(\frac{w_z^2}{1+v_z}+\frac{v_z^2v_{zz}^2w^2}{(1+v_z^2)^3}\right)\frac{v^{n-1}}{\sqrt{1+v_z^2}}\\
&+\left(\frac{(n-1)v_z^2v_{zz}w^2}{v(1+v_z^2)^2}+\frac{v_{zz}^2w^2}{(1+v_z^2)^2}+\frac{v_zv_{zzz}w^2}{(1+v_z^2)^2}-\frac{5v_z^2v_{zz}^2w^2}{(1+v_z^2)^3}\right)\frac{v^{n-1}}{\sqrt{1+v_z^2}}.
\end{align*}
Now we differentiate $\frac{-v_{zz}}{(1+v_z^2)^{\frac32}}+\frac{n-1}{v\sqrt{1+v_z^2}}=constant$ with respect to $z$ to obtain:
\[
\frac{-v_{zzz}}{(1+v_z^2)^{\frac32}}+\frac{3v_zv_{zz}^2}{(1+v_z^2)^{\frac52}}-\frac{(n-1)v_z}{v^2\sqrt{1+v_z^2}}-\frac{(n-1)v_zv_{zz}}{v(1+v_z^2)^{\frac32}}=0,
\]
so that
\begin{align*}
\frac{u_z^2u^{n-1}}{\sqrt{1+v_z^2}}=&\frac{\pa}{\pa z}\left(\frac{-v^{n-1}v_zv_{zz}w^2}{(1+v_z^2)^{\frac52}}\right)+\left(\frac{w_z^2}{1+v_z}-\frac{v_z^2v_{zz}^2w^2}{(1+v_z^2)^3}\right)\frac{v^{n-1}}{\sqrt{1+v_z^2}}\\
&+\left(\frac{v_{zz}^2w^2}{(1+v_z^2)^2}-\frac{(n-1)v_z^2w^2}{v^2(1+v_z^2)}\right)\frac{v^{n-1}}{\sqrt{1+v_z^2}}\\
=&\frac{\pa}{\pa z}\left(\frac{-v^{n-1}v_zv_{zz}w^2}{(1+v_z^2)^{\frac52}}\right)+\left(\frac{w_z^2}{1+v_z}+\frac{v_{zz}^2w^2}{(1+v_z^2)^3}-\frac{(n-1)v_z^2w^2}{v^2(1+v_z^2)}\right)\frac{v^{n-1}}{\sqrt{1+v_z^2}}.
\end{align*}

Using that for a rotationally symmetric hypersurface $|A|^2=\frac{n-1}{v^2(1+v_z^2)}+\frac{v_{zz}^2}{(1+v_z^2)^3}$,
\begin{align*}
\tilde{J}_v(w)=&\frac{\int_{\SS^{n-1}}\int_0^d \left(\frac{w_z^2}{1+v_z}-\frac{(n-1)(v_z^2+1)w^2}{v^2(1+v_z^2)}+\frac{1}{v^2}\|\tilde{\nabla} w\|_{\SS^{n-1}}^2\right)\frac{v^{n-1}}{\sqrt{1+v_z^2}}\,dz\,d\mu_{\SS^{n-1}}}{\int_{\SS^{n-1}}\int_0^d w^2\frac{v^{n-1}}{\sqrt{1+v_z^2}}\,dz\,d\mu_{\SS^{n-1}}}\\
=&\frac{\int_{\SS^{n-1}}\int_0^d \left(\|\nabla w\|_{\Si}^2-\frac{(n-1)w^2}{v^2} \right)\frac{v^{n-1}}{\sqrt{1+v_z^2}}\,dz\,d\mu_{\SS^{n-1}}}{\int_{\SS^{n-1}}\int_0^d w^2\frac{v^{n-1}}{\sqrt{1+v_z^2}}\,dz\,d\mu_{\SS^{n-1}}}.
\end{align*}
\end{proof}


\section{Existence of the Inverse Projection}\label{PsiSec}
In this section we prove Lemma \ref{PsiExist}. Fix $t>0$ and define $\Phi:X_1\times\RR\times X_t\to X_1\times\RR$ such that
\[
\Phi(\bar{u},r,u)=\left(W(1)[u]-\bar{u},\om_n\int_0^du(z)^n\,dz-V(r)\right)
\]
and consider it's linearisation with respect to $u$:
\[
D_u\Phi(\bar{u},r,u)[w]=\left(W(1)[w],n\om_n\int_0^dw(z)u(z)^{n-1}\,dz\right)
\]
If we evaluate at $(\bar{v}(\cdot;t),t,v(\cdot;t))$, noting that $\Phi(\bar{v}(\cdot;t),t,v(\cdot;t))=(0,0)$, we obtain
\[
D_u\Phi(\bar{v}(\cdot;t),t,v(\cdot;t))[w]=\left(W(1)[w],n\om_n\int_0^dv(z;t)^{n-1}\,dz(I-W(t))[w]\right).
\]
If $w$ is in the null space of $D_u\Phi(\bar{v}(\cdot;t),t,v(\cdot;t))$ then:
\[
W(1)[w]=0\ \text{and}\ (I-W(t))[w]=0.
\]
Therefore $w=(I-W(t))[w]+W(t)[w]=0+W(t)[W(1)[w]]=W(t)[0]=0$. So, $D_u\Phi(\bar{v}(\cdot;t),t,v(\cdot;t))$ is a Banach space isomorphism for any $t>0$ and the existence of $\psi_t$ follows from the implicit function theorem.

The formula for the derivative follows by taking the derivative of $\tilde{\Phi}_t(\bar{u},r)=\Phi(\bar{u},r,\psi_t(\bar{u},r))=(0,0)$ with respect to $\bar{u}$. That is:
\[
\left(W(1)[D_{\bar{u}}\psi_t(\bar{u},r)[\bar{w}]]-\bar{w},n\om_n\int_0^dD_{\bar{u}}\psi_t(\bar{u},r)[\bar{w}]\psi_t(\bar{u},r)^{n-1}\,dz\right)=(0,0).
\]
By the first condition $D_{\bar{u}}\psi_t(\bar{u},r)[\bar{w}]=\bar{w}+C$, and substitution into the second condition gives
\[
\int_0^d\bar{w}\psi_t(\bar{u},r)^{n-1}\,dz+C\int_0^d\psi_t(\bar{u},r)^{n-1}\,dz=0,
\]
resulting in the formula given.

\bibliographystyle{plain}

\end{document}